\documentclass{article}

\usepackage{amsmath,amssymb,amsthm,mathrsfs}

\newcommand{\BC}{{\mathbb C}}\newcommand{\BD}{{\mathbb D}}

\newcommand{\BN}{{\mathbb N}}

\newcommand{\BT}{{\mathbb T}}

\newcommand{\cD}{{\mathcal D}}
\newcommand{\cF}{{\mathcal F}}

\newcommand{\cM}{{\mathcal M}}

\newcommand{\cX}{{\mathcal X}}
\newcommand{\cY}{{\mathcal Y}}

\newcommand{\wtilG}{\widetilde{G}}

\newcommand{\ga}{\gamma}\newcommand{\Ga}{\Gamma}
\newcommand{\De}{\Delta}

\newcommand{\Th}{\Theta}

\newcommand{\la}{\lambda}\newcommand{\La}{\Lambda}

\newcommand{\om}{\omega}

\newcommand{\rank}{\textup{rank\,}}

\newcommand{\im}{\textup{Im\,}}

\newcommand{\kr}{\textup{Ker\,}}

\newcommand{\spec}{r_\textup{spec}}
\newcommand{\codim}{\textup{codim}\,}

\newcommand{\ov}[1]{{\overline{#1}}}

\newcommand{\tu}[1]{\textup{#1}}

\newcommand{\ands}{\quad\mbox{and}\quad}

\newtheorem{theorem}{Theorem}[section]
\newtheorem{corollary}[theorem]{Corollary}
\newtheorem{lemma}[theorem]{Lemma}
\newtheorem{proposition}[theorem]{Proposition}

\newcommand{\ts}{{\times}}
\newcommand{\iy}{{\infty}}
\newcommand{\va}{{\varphi}}

\begin{document}

\title{Optimal solutions to matrix-valued Nehari problems and related limit theorems}

\author{A.E. Frazho, S. ter Horst and M.A. Kaashoek}

\date{}

\maketitle

\begin{abstract}
In a 1990 paper Helton and Young showed that under certain conditions the optimal solution of the Nehari problem corresponding to a finite rank Hankel operator  with scalar entries  can be efficiently approximated by certain functions defined in terms of finite dimensional restrictions of the  Hankel operator. In this paper it is shown  that these approximants appear as optimal solutions to restricted Nehari problems. The latter problems  can be solved using relaxed commutant lifting theory.  This observation is used to extent the Helton and Young approximation result to a matrix-valued setting.  As in the Helton and Young paper the rate of convergence  depends on the choice of the initial space in the approximation scheme.
\end{abstract}




\section{Introduction}\label{S:intro}

Since the 1980s, the Nehari problem played an important role in system and control theory, in particular, in the $H^\infty$-control solutions to sensitivity minimization and robust stabilization, cf., \cite{Francis}.  {In system and control theory the Nehari problem} appears mostly as a distance problem: Given $G$ in $L^\infty$, determine the distance of $G$ to $H^\infty$, that is, find the quantity
$d:=\inf\{\|G-F\|_\iy\mid  F\in H^\iy\}$ and, if possible, find an $F\in H^\iy$ for which this infimum is attained. Here all functions are complex-valued functions on the unit circle $\BT$. It is well-known that the solution to this problem is determined by the Hankel operator $H$ which maps $H^2$ into $K^2=L^2\ominus H^2$ according to the  rule $H f= P_- (Gf)$, where $P_-$ is the orthogonal projection of  $L^2$ onto $K^2$. Note that $H$ is uniquely determined by the Fourier coefficients of $G$ with negative index. Its operator norm determines the minimal distance.  In fact, $d=\|H\|$ and the infimum is attained. Furthermore, if $H$ has a maximizing vector $\va$, that is, $\va$ is a non-zero function in $H^2$ such that $\|H\va\|=\|H\|\,\|\va\|$, then the  {AAK theory \cite{AAK68, AAK71} (see also \cite{DS67})} tells us that the  best approximation $\widehat{G}$ of $G$ in $H^\iy$ is unique and is given by
\begin{equation}\label{bestapp1}
\widehat{G}(e^{it})=G(e^{it})-\frac{(H \va)(e^{it})}{\va (e^{it})}\quad\mbox{a.e.}
\end{equation}

{By now the connection between the Nehari problem and Hankel operators is well established, also for  matrix-valued and operator-valued functions, and has been put into the larger setting of metric constrained interpolation problems, see, for example, the books \cite[Chapter IX]{FF90}, \cite[Chapter XXXV]{GGK93},  \cite[Chapter I]{FFGK98}, \cite[Chapter 5]{Peller03} and \cite[Chapter 7]{AD08}, and the references therein.}

{The present paper is inspired by Helton-Young \cite{HY90}. Note that}  formula \eqref{bestapp1} and the maximizing vector $\va$, may be hard to compute, especially if $H$ has large or infinite rank. Therefore, to approximate the optimal solution \eqref{bestapp1}, Helton-Young \cite{HY90} replaces $H$ by the restriction $\dot{H}=H|_{H^2\ominus z^n qH^2}$ to arrive at
\begin{equation}\label{bestapp2}
\widetilde{G}(e^{it})=G(e^{it})-\frac{(\dot{H}\widetilde{\va})(e^{it})}{\widetilde{\va} (e^{it})},\quad\mbox{a.e.}
\end{equation}
as an approximant of $\widehat G$. Here $n$ is a positive integer, $q$ is a polynomial and $\widetilde{\va}$ is a maximizing vector of $\dot{H}$. Note that a maximizing vector $\widetilde{\va}$ of $\dot{H}$ always exists, since $\rank \dot{H}\leq n+\deg q$, irrespectively of the rank of $H$ being finite, or not.

In \cite{HY90} it is shown that $\widetilde{G}$ is a computationally efficient approximation of the optimal solution $\widehat{G}$ when the zeros of the polynomial $q$ are close to the poles of $G$ in the open unit disk $\BD$ that are close to the unit circle $\BT$. To be more precise, it is shown that if $G$ is rational, i.e., $\rank H<\infty$, and $\|H\|$ is a simple singular value of $H$, then $\|\widehat{G}-\wtilG\|_\infty$ converges to 0 as $n\to\infty$. This convergence is proportional to $r^n$ if the poles of $G$ in $\BD$ are within the disc $\BD_r=\{z\in\BC\mid |z|<r\}$, and the rate of convergence can be improved by an appropriate choice of the polynomial~$q$.

It is well-known that the Nehari problem fits in the commutant lifting framework, and that the solution formula \eqref{bestapp1} follows as a corollary of the commutant lifting theorem. {We shall see that the same holds true for formula \eqref{bestapp2} provided one uses   the relaxed commutant lifting framework of \cite{FFK02}; cf., Corollary 2.5 in \cite{FFK02}.}

To make the connection with relaxed commutant lifting more precise, define $R_n$ to be the orthogonal projection of $H^2$ onto  $H^2 \ominus z^{n-1}qH^2$, and  put $Q_n=SR_n$, where $S$ is the forward shift on $H^2$. Then the operators $R_n$ and $Q_n$ both map $H^2$ into $H^2\ominus z^nqH^2$, and the restriction operator $H_{n}:=H|_{H^2\ominus z^n qH^2}$ satisfies the intertwining relation $V_- H_{n} R_n=H_{n} Q_n$. Here $V_-$ is the compression of the forward shift $V$ on $L^2$ to $K^2$. Given this intertwining relation,  the relaxed commutant lifting theorem \cite[Theorem 1.1]{FFK02} tells us that there exists an operator $B_n$
from $H^2\ominus z^nqH^2$ into $L^2$ such that
\begin{equation}\label{rcl2}
P_-B_n= H_{n}, \quad  VB_nR_n=BQ_n, \quad \|B_n\|=\|H_{n}\|.
\end{equation}
The second identity in \eqref{rcl2} implies (see Lemma \ref{L:symbol} below) that
for a solution $B_n$ to \eqref{rcl2} there exists a unique function $\Phi_n\in L^2$ such that the action of $B_n$ is given by
\begin{equation}\label{symbol}
(B_nh)(e^{it})=\Phi_n(e^{it})h(e^{it})\hspace{.2cm} a.e.\quad (h\in H^2\ominus z^nqH^2).
\end{equation}
Furthermore,  {since $H_n$ has finite rank}, there exists only one solution $B_n$ to \eqref{rcl2} (see Proposition \ref{P:VectorSol} below), and if $\psi_n=\widetilde{\va}$ is a maximizing vector of  $H_{n}$, then this unique solution is given by \eqref{symbol} with $\Phi_n$ equal to
\begin{equation}\label{uniqueSol}
\Phi_n(e^{it})=\frac{(H_{n}\psi_n)(e^{it})}{\psi_n(e^{it})}
=\frac{(\dot{H}\widetilde{\va})(e^{it})}{\widetilde{\va}(e^{it})},\quad\mbox{a.e.}.
\end{equation}
Thus $G-\widetilde{G}$ appears as an optimal solution to a relaxed commutant lifting problem.

{This observation together with  the relaxed commutant lifting theory developed in the last decade, enabled us to extent the Helton-Young convergence result for optimal solutions  {in} \cite{HY90} to a matrix-valued setting, that is, to derive} an analogous convergence result for optimal solutions to matrix-valued Nehari problems; see Theorem \ref{T:main2} below. A complication in this endeavor is that formula \eqref{bestapp1} generalizes to the vector-valued case, but not to the matrix-valued case. Furthermore, in the matrix-valued case there is in general no unique solution. We overcome the latter complication by only considering {the central solutions which} satisfy an additional maximum entropy-like condition.  {On the way we also derive explicit state space formulas for  optimal solutions to the classical and restricted Nehari problem assuming that the Hankel operator is of finite rank and satisfies an appropriate condition on  the space spanned by its maximizing vectors. These state space formulas play an essential role in the proof of the convergence theorem.}

This paper consists of 6 sections including  the present introduction. In Section~\ref{secRestrN}, which  has a preliminary character, we introduce a restricted version of the  matrix-valued Nehari problem,  and  {use  relaxed commutant lifting theory to show that it always has an optimal solution. Furthermore, again using   relaxed commutant lifting theory, }we derive a formula for the (unique) central optimal solution.  In Section \ref{S:MainRes} we state our main convergence result. In Section \ref{S:centropt2} the formula for the (unique) central optimal solution derived in Section \ref{secRestrN} is developed further, and {in Section \ref{S:NehariCase} this formula is specified  for the classical Nehari problem. Using these formulas Section \ref{secprfmthm} presents the proof of the main convergence theorem.}

\medskip\noindent \textbf{Notation and terminology.} We conclude this introduction with a few words about notation and terminology.
 Given $p,q$ in $\BN$, the set of positive integers, we write $L^2_{q\ts p}$ for the space of all  $q\ts p$-matrices with entries in $L^2$, the Lebesgue space of square integrable functions on the unit circle. Analogously, we write $H^2_{q\ts p}$ for the space of all  $q\ts p$-matrices with entries in the classical Hardy space $H^2$, and $K^2_{q\ts p}$ stands for the space of all  $q\ts p$-matrices with entries in the  space $K^2=L^2\ominus H^2$, the orthogonal compliment of  $H^2$ in $L^2$. Note that each $F\in L^2_{q\ts p}$ can be written uniquely as a sum $F=F_+ +F_-$ with $F_+\in H^2_{q\ts p}$ and  $F_-\in K^2_{q\ts p}$. We shall refer to $F_+$ as the \emph{analytic part} of   $F$ and to $F_-$ as its \emph{co-analytic part}. When {there is only one column}  we simply write $L^2_p$, $H^2_p$ and $K^2_p$ instead of  $L^2_{p\ts 1}$, $H^2_{p\ts 1}$ and $K^2_{p\ts 1}$. Note that $L^2_p$, $H^2_p$ and $K^2_p$ are Hilbert spaces and $K^2_p=L^2_p\ominus H^2_p$.  Finally, $L^\iy_{q\ts p}$ stands for the space of all  $q\ts p$-matrices whose  entries are essentially bounded on the unit circle with respect to the Lebesque measure, and $H^\iy_{q\ts p}$ stands for the space of all  $q\ts p$-matrices whose  entries are analytic and uniformly bounded on the open unit disc $\BD$. Note that each $F\in L^\iy_{q\ts p}$ belongs to $L^2_{q\ts p}$  and hence  the analytic part $F_+$ and the co-analytic part $F_-$ of $F$ are well defined. These functions belong to $L^2_{q\ts p}$ and it may happen that neither $F_+$ nor $F_-$ belong to $L^\iy_{q\ts p}$.  In the sequel we shall  need the following embedding and projection operators:
\begin{align}
&E:\BC^p\to H_p^2, \quad  Eu(\la)=u\quad (z\in \BD);\label{defE}\\
&\Pi : K_q^2\to  \BC^q, \quad \Pi f= \frac{1}{2\pi}\int_0^ {2\pi}e^{-it}f(e^{it})\,dt. \label{defPi}
\end{align}
Throughout  $G\in L^\infty_{q\ts p}$, and $H: H_p^2\to K_q^2$ is the Hankel operator defined by the co-analytic part of $G$, that is, $H f=P_-(Gf)$ for each $f\in H_p^2$.  Here $P_-$ is the orthogonal projection of $L_q^2$ onto $ K_q^2$.   Note that $V_- H=HS$, where  $S$ is the forward shift on  $H_p^2$ and $V_-$ is the compression to $K_q^2$ of the forward shift $V$ on $L_q^2$.

Finally, we  associate with the Hankel operator $H$  two auxiliary operators involving the closure of its range, i.e., the space $\cX= \overline{\im H}$, as follows:
\begin{align}
& Z:\cX\to\cX, \qquad Z=V_-|_\cX,\label{defZ}\\
& W:H_p^2 \to\cX, \qquad Wf=Hf\quad (f\in H_p^2). \label{defW}
\end{align}
Note that  $\cX:=\overline{\im H}$ is a $V_-$-invariant subspace of $K_q^2$.  Hence $Z$ is a well-defined  contraction. Furthermore, if $\rank H$ is finite, then the spectral radius  $r_{\textup{spec}}(Z)$  is strictly less than one and the co-analytic part $G_-$ of $G$  is the rational matrix function given by
\[
G_-(\la)= (\Pi|_\cX) (\la I- Z)^{-1}WE.
\]
In system theory the  right hand side of the above identity is known as the  restricted
backward shift realization of $G_-$; see, for example,  \cite[Section 7.1]{CF03}.  This realization is minimal, and hence the eigenvalues of  $Z$ coincide with the poles of  $G_-$ in $\BD$. {In particular, $\spec(Z)<1$.} Since $V_-H=HS$, we have $ZW=WS$. Furthermore, $\kr H^*=K_q^2\ominus \cX$.

\setcounter{equation}{0}
\section{Restricted Nehari problems and relaxed commutant lifting} \label{secRestrN}
In this section we introduce a restricted version of the Nehari problem, and we prove that it is equivalent to a certain  relaxed commutant lifting problem. Throughout $\cM$ is a subspace of $H^2_p$ such that
\begin{equation}\label{cM1}
 S^*\cM \subset \cM, \qquad \kr S^*\subset \cM.
\end{equation}
With $\cM$ we associate  {operators $R_\cM$ and $Q_\cM$ acting on $H^2_p$, both}  mapping $H^2_p$ into $\cM$. By definition $R_\cM$ is the orthogonal projection of $H^2_p$ onto $S^*\cM$ and $Q_\cM=SR_\cM$.

We begin by introducing the notion of an $\cM$-norm. We say that $\Phi\in L^2_{q\ts p}$ has a \emph{finite $\cM$-norm} if $\Phi h\in L^2_{q}$ for each $ h\in \cM$  and the map $h\mapsto \Phi h$ is a bounded linear operator, and in that case we define
\[
\|\Phi\|_\cM=\sup \{\|\Phi h\|_{L^2_{q}} \mid h\in \cM, \quad \|h\|_{H^2_p}\leq 1\}.
\]
If  $\cM$ is finite dimensional, then each $\Phi\in L^2_{q\ts p}$ has a finite $\cM$-norm. Furthermore,  $\Phi\in L^\iy_{q\ts p}$  has a finite $\cM$-norm for every choice of $\cM$, and in this case $\|\Phi\|_\cM\leq \|\Phi\|_\iy$, with equality if $\cM=H_p^2$.  Note that  $\Phi\in L^2_{q\ts p}$ has a finite $\cM$-norm and $G\in L^\iy_{q\ts p}$ imply $G-\Phi$ has a finite $\cM$-norm.

We are now ready to formulate  the $\cM$-restricted Nehari problem.  Given $G\in L^\iy_{q\ts p}$ and a subspace $\cM$ of $H^2_p$, we define the \emph{optimal $\cM$-restricted Nehari problem} to be the problem of determining the quantity
\begin{equation}\label{inf}
d_\cM:= \inf \{\|G-F\|_\cM \mid  \mbox{$ F\in H^2_{q\ts p}$ and $F$ has a finite $\cM$-norm}\},
\end{equation}
and, if possible, to find a function $F\in H^2_{q\ts p}$  of finite $\cM$-norm at which the infimum is attained. In this case, a function $F$ attaining the infimum is called  an \emph{optimal solution}. The suboptimal variant of the problem allows  the norm $\|G-F\|_\cM $ to be  larger than the infimum. When  $\cM= H^2_p$, the problem  coincides with the classical matrix-valued Nehari problem in $L^\iy_{q\ts p}$.  {In \cite{tH07, tH09} the} case where $\cM=H^2_p \ominus S^k H^2_p$, with $k\in\BN$, was considered.

\begin{proposition}\label{propRestrN} Let $G\in L^\iy_{q\ts p}$, and let $\cM$ be a subspace of $H^2_p$  satisfying the conditions in \eqref{cM1}. Then the $\cM$-restricted Nehari problem has an optimal solution and the quantity $d_\cM$ in
\eqref{inf} is equal to $\ga_\cM:=\|H|_\cM\|$, where $H: H_p^2\to K_q^2$ is the Hankel operator defined by the co-analytic part of $G$.
 \end{proposition}

We shall derive the above result as a corollary to  the relaxed commutant lifting theorem \cite[Theorem 1.1]{FFK02}, in a way similar to the way one proves the Nehari theorem  using the classical commutant lifting theorem (see, for example,  \cite[Section~II.3]{FF90}).  For this purpose we need the following notion. We say that an operator $B$ from  $\cM$ into $L^2_q$ is \emph{defined by} a $\Phi\in L^2_{q\ts p}$ if the action of $B$ is given by
\begin{equation}\label{actB}
(Bh)(e^{it})=\Phi(e^{it})h(e^{it})\hspace{.2cm}  a.e. \quad(h\in \cM).
\end{equation}
In that case, $\Phi$ has a finite $\cM$-norm, and  $\|\Phi\|_\cM=\|B\|$.   When \eqref{actB} holds we refer to $\Phi$ as the \emph{defining function} of $B$. The following lemma characterizes operators $B$ from  $\cM$ into $L^2_q$  {defined by} a function $\Phi\in L^2_{q\ts p}$ in terms of an  intertwining relation.

\begin{lemma}\label{L:symbol}
Let $\cM$ be a subspace of $H^2_p$ satisfying \eqref{cM1}, and let $B$ be a
bounded operator from $\cM$ into $L^2_q$. Then $B$ is defined by a $\Phi\in L^2_{q\ts p}$  if and only if $B$ satisfies the intertwining relation $VBR_\cM=BQ_\cM$.  In that case, $\Phi(\cdot)u=BEu(\cdot)$ for any $u\in \BC^p$ and
$\|B\|=\|\Phi\|_\cM$\end{lemma}

\begin{proof}[\bf Proof]
This result follows by a modification of the proof of Lemma 3.2 in \cite{FtHK08}.
We omit the details.\end{proof}

\begin{proof}[\bf Proof of Proposition \ref{propRestrN}]
Put $\ga_\cM=\|H|_\cM\|$. Recall that the Hankel operator $H$ satifies the intertwining relation $V_- H=HS$. This implies  $V_-H|_\cM R_\cM=H|_\cM Q_\cM$. Here $R_\cM$ and $Q_\cM$ are the operators defined in the first paragraph of the present section. Since $Q_\cM^* Q_\cM =  R_\cM^* R_\cM$ and $V$ is an isometric lifting of $V_-$, the quintet
\begin{equation}\label{liftset}
 \{H|_\cM, V_-, V, R_\cM, Q_\cM, \ga_\cM\}
\end{equation}
is a lifting data set  in the sense of  Section 1 in \cite{FFK02}. Thus Theorem 1.1 in \cite{FFK02} guarantees  the existence of an operator $B$ from $\cM$ into $L^2_q$ with the properties
\begin{equation}\label{rcl3}
P_-B=H|_\cM,\quad V BR_\cM=BQ_\cM, \quad \|B\|= \ga_\cM.
\end{equation}
By Lemma \ref{L:symbol} the second equality in \eqref{rcl3} tells us there exists a $\Phi\in L^2_{q\ts p}$ defining $B$, that is,  the action of  $B$ is given by \eqref{actB}. As $\Phi(\cdot)u=BEu(\cdot)$, the first identity in \eqref{rcl3} shows that $G_-=\Phi_-$, and hence $F:=G-\Phi\in H_{q\ts p}^2$.  Furthermore,
\[
\|G-F\|_\cM=\|\Phi\|_\cM=\|B\|=\ga_\cM,
\]
because of the third identity in \eqref{rcl3}. Thus the quantity $d_\cM$ in  \eqref{inf} is less than or equal to $\ga_\cM$.

It remains to prove that $d_\cM\geq \ga_\cM$. In order to do this, let $\tilde{F} \in H^2_{q\ts p}$ and have a finite $\cM$-norm. Put $\tilde{\Phi}=G-\tilde{F}$. Then $\tilde{\Phi}$ has a finite $\cM$-norm. Let $\tilde{B}$ be the operator from $\cM$ into $L^2_q$ defined by $\tilde{\Phi}$. Since  $\tilde{F} \in H^2_{q\ts p}$, we have $G_-=\tilde{\Phi}_-$, and hence the first identity in \eqref{rcl3} holds with $\tilde{B}$ in place of $B$. It follows that
\[
\|G-\tilde{F}\|_\cM=\|\tilde{\Phi}\|_\cM=\|\tilde{B}\|\geq \|H|_\cM\|=\ga_\cM.
\]
This completes the proof. \end{proof}

\medskip
In the scalar case,  or more generally in the case when $p=1$, the optimal solution is unique. Moreover  this unique solution is given  by a formula analogous to \eqref{bestapp2};  { cf., \cite{AAK68}}. This is the contents of the next proposition which is proved in much the same way as the corresponding result for the Nehari problem. We omit the details.

\begin{proposition}\label{P:VectorSol}
Assume $p=1$, that is, $G\in L^\infty_q$ and $\cM$ a subspace of $H^2$ satisfying \eqref{cM1}. Assume that $H|_\cM$ has a maximizing vector $\psi\in\cM$. Then there exists only one optimal solution $F$ to the $\cM$-restricted Nehari problem \eqref{rcl3}, and this solution is given by
\begin{equation}\label{uniqueSymbol}
F(e^{it})=G(e^{it})- \frac{(H\psi)(e^{it})}{\psi(e^{it})}\  a.e.
\end{equation}
\end{proposition}

In general, if $p>1$ the optimal solution is not unique. To deal with this non-uniqueness, we shall single out a particular optimal solution.

First note that the proof of Proposition \ref{propRestrN} shows that there is a one-to-one correspondence between the optimal solutions of the $\cM$-restricted Nehari problem of $G$ and  all interpolants for $H|_\cM$ with respect to the lifting data set \eqref{liftset}, that is, all operators $B$ from $\cM$ into $L^2_q$ satisfying \eqref{rcl3}. This correspondence is given by
\begin{equation}\label{corresp}
 B\mapsto F=G-\Phi, \ \mbox{where $\Phi$ is the defining function of $B$}.
\end{equation}

Next we use that the relaxed commutant lifting theory tells us that  among  all interpolants for $H|_\cM$ with respect to the lifting data set \eqref{liftset} there is a particular one, which is called the central interpolant for $H|_\cM$ with respect to the lifting data set \eqref{liftset}; see \cite[Section 4]{FFK02}.  This central interpolant  is uniquely determined by a maximum entropy principle (see \cite[Section 8]{FFK02}) and  given by an explicit formula using the operators appearing in the lifting data set.

Using  the correspondence \eqref{corresp} we say that an optimal solution $F$ of the  $\cM$-restricted Nehari problem of $G$  is the \emph{central optimal solution} whenever $\Phi:=G-F$ is the defining function of the central interpolant  $B$   for $H|_\cM$ with respect to the lifting data set \eqref{liftset}. Furthermore, using the formula given in \cite[Section 4]{FFK02} for the central interpolant the correspondence \eqref{corresp} allows us to derive a formula for the central optimal solution. To state this formula we need to make some preparations.

As before   $\ga_\cM =\|H|_\cM\|$. Note that  {$\|HP_\cM S\|\leq\|HP_\cM\|=\|H|_\cM\|$}, where $P_\cM$ is the orthogonal projection of $H^2(\BC^p)$ on $\cM$.  This allows us to define  the following defect operators acting on $H^2(\BC^p)$
\begin{align}
D_\cM&=(\ga_\cM^2 I-P_\cM H^*HP_\cM)^{1/2} \  \mbox{on $H^2(\BC^p)$},
\label{defect1}\\
D_\cM^\circ&=(\ga_\cM^2 I-S^*P_\cM H^*HP_\cM S)^{1/2} \  \mbox{on $H^2(\BC^p)$}.\label{defect2}
\end{align}
For later purposes we note that $S^*D_\cM^2 S= D_\cM^{\circ 2}$.  Next define
\begin{align}
 &\om=
\begin{bmatrix}
 \om_1\\\om_2
\end{bmatrix}: H^2_p\to
\begin{bmatrix}
 \BC^q\\H^2_p
\end{bmatrix},\label{defoma}
\\
&\om(D_\cM Q_\cM)=
\begin{bmatrix}
\Pi H R_\cM \\ D_\cM  R_\cM
\end{bmatrix}  \ands \om|_{ \kr Q_\cM^*D_\cM}=0.\label{defomb}
\end{align}
{From the relaxed commutant lifting theory   we know that $\om$ is a well defined partial isometry with initial space $\cF=\overline{\im D_\cM Q_\cM}$. Furthermore, the forward shift operator $V$ on $L^2_q$ is the Sz.-Nagy-Sch\"affer isometric lifting of $V_-$. Then as a consequence of \cite[Theorem  4.3]{FFK02} and the above analysis we obtain the following result.}

\begin{proposition}\label{propforopt1} Let $G\in L^\iy_{q\ts p}$, and let $\cM$ be a subspace of $H^2_p$  satisfying the conditions in \eqref{cM1}. Then the central optimal solution $F_\cM$ to the $\cM$-restricted Nehari problem is given by $F_\cM=G-\Phi_\cM$, where $\Phi_\cM\in L^2_{q\ts p}$  has finite $\cM$-norm, the co-analytic part of $\Phi_\cM$ is equal to $G_-$, and  the analytic {part $\Phi_{\cM, +}$} of $\Phi_\cM$ is given by
\begin{equation}\label{eqPhiM+}
\Phi_{\cM, +}(\la)=\om_1(I-\la \om_2)^{-1}D_\cM E.
\end{equation}
Here $E$ is defined by \eqref{defE}, and $\om_1$ and $\om_2$ are  defined by \eqref{defoma} and \eqref{defomb}.
\end{proposition}

It is this  central optimal solution $F_\cM$ we shall be working with. {From Corollary~4.4 in \cite{FFK02}  (see also \cite[Theorem 1.1]{FtHK06}) we know that $\cF=\overline{\im D_\cM Q_\cM}=\cD_\cM$ implies that the central solution of \eqref{rcl3}  is the only optimal solution to the  $\cM$-restricted Nehari problem. The latter  fact will play a role in Section \ref{S:centropt2}.}

\setcounter{equation}{0}
\section{Statement of the main convergence result}\label{S:MainRes}


Let $G\in L^\iy_{q\ts p}$, and let $H$ be the Hankel operator defined by the co-analytic part of $G$.  In our main approximation result we  shall assume that the following two conditions are satisfied:
\begin{itemize}
\item[(C1)] $H$ has finite rank,
\item[(C2)] none of  the maximizing vectors of $H$ belongs $SH_p^2$, and  the space spanned by the maximizing vectors of $H$ has dimension $p$.
\end{itemize}
{Note that (C1) is equivalent to $G$ being the sum of a rational matrix function with all its poles in $\BD$ and a matrix-valued $H^\iy$ function. }

{In the scalar case the second part of (C2) implies the first part. To see this let $p=q=1$, and assume that the space spanned by the maximizing vectors of $H$ is one dimensional.  Let $Sv$ be  a maximizing vector of $H$. Since $S$ is an isometry and $V_-H=HS$, we have $v\not =0$ and
\[
 \|H\|  \|v\|= \|H\|  \|Sv\|= \|HSv\|= \|V_-Hv\|\leq \|Hv\|\leq  \|H\|  \|v\|.
\]
Thus the inequalities are equalities, and $v$ is a maximizing vector of $H$. As the the space spanned by the maximizing vectors of $H$ is assumed to be one dimensional, $v$ must be a scalar multiple of $Sv$, which can only happen when $v=0$, which contradicts  $v\not =0$. Thus the first part of (C2) is fulfilled. Next observe that for $p=q=1$ the statement ``the space spanned by the maximizing vectors of $H$ has dimension one'' is  just equivalent to the requirement  that $\|H\|$ is  a simple singular value of $H$, which is precisely the condition used in Theorem 2 of the Helton-Young paper \cite{HY90}. }

 As we shall see in Section \ref{S:NehariCase} the two conditions (C1)  and (C2)  guarantee that the solution to the optimal Nehari problem  is unique.

 For our approximation scheme we fix a finite dimensional subspace $\cM_0$ of $H^2_p$ invariant under $S^*$, and we define recursively
\begin{equation}\label{cMk}
\cM_{k}=\kr S^*\oplus S\cM_{k-1},\quad k\in \BN.
\end{equation}
{Since $\cM_0$ is invariant under $S^*$, the space $\cM_0^\perp$ is invariant under $S$, and the Beurling-Lax theorem tells us that  $ \cM_0^\perp=\Th H^2_\ell$, where $\Th\in H^\iy_{p\ts \ell}$ and can be taken to be inner. Using this representation one checks that $\cM_k=H^2_p\ominus z^k\Th H^2_\ell$ for each $k\in \BN$. It follows that $\cM_0\subset \cM_1\subset  \cM_2\subset \cdots$ and $\bigvee_{k\geq 0}\cM_k=H^2_p$. Furthermore,}
\begin{equation}\label{Mconditions}
   S^*\cM_k\subset \cM_k \ands \kr S^*\subset \cM_k, \quad k\in \BN.
\end{equation}
Note that  the spaces $\cM_k=H^2\ominus z^kqH^2$, $k=1,2,\ldots$, appearing in \cite{HY90} satisfy \eqref{cMk} with $\cM_0=H^2\ominus qH^2$.

\begin{theorem}\label{T:main2}
Let $G\in L^\iy_{q\ts p}$. Assume that conditions $(C1)$ and $(C2)$ are  satisfied, and  let the sequence of subspaces $\{\cM_k\}_{k\in\BN}$ be defined by \eqref{cMk} with $\cM_0$ a finite dimensional $S^*$-invariant subspace of $H^2_p$. Let ${F}$ be the  unique optimal solution  to the Nehari problem for $G$, and  for each $k\in\BN_+$ let $F_k$  be the central  optimal  solution to  the $\cM_k$-restricted Nehari problem.  Then   $G-{F}$ is  a rational function in $H^\iy_{q\ts p}$, and  for $k\in\BN_+$ sufficiently large,   the same holds true for   $G-F_k$. Furthermore,  $\|F_k-{F}\|_{\iy}\to 0$ for $k\to \iy$.
More precisely, if all the poles of $G$ inside $\BD$ are within the disk $\BD_r=\{\la\mid |\la|<r\}$, for $r<1$, then there exists a number $L>0$ such
that $\|F_k-\hat{F}\|_{\iy}< Lr^k$ for $k$ large enough.
\end{theorem}


Improving the rate of convergence is one of the main issues in  \cite{HY90}, where it is shown that for the case when  the poles of $G$ inside $\BD$ are close to the unit circle, that is, $r$ close to $1$, convergence with $\cM_0=\{0\}$ may occur at a slow rate. In  \cite{HY90} it is also shown how to choose (in the scalar case) a scalar polynomial $q$ so that   the choice $\cM_0=H^2\ominus qH^2$  increases the rate of convergence.  {In fact, if the roots of $q$ coincide with the poles of $G$ in $\BD_r\backslash \BD_{_0}$, then starting with $\cM_0=H^2\ominus qH^2$ the convergence is of order $O(r_0^k)$  rather than $O(r^k)$.} In Section \ref{secprfmthm} we shall see that  Theorem \ref{T:main2} remains true if $r<1$ is larger than the spectral radius of the operator $V_-|_{H\cM_0^\perp}$, and thus  {again} the convergence rate can be improved by an appropriate choice of $\cM_0$. To give a trivial example: when $\cM_0$ is chosen in such a way that it includes $\im H^*$, all the central optimal solutions $F_k$  in Theorem \ref{T:main2} coincide with the unique optimal solution  solution $F$ to the Nehari problem.

\setcounter{equation}{0}
\section{The central optimal  solution revisited}\label{S:centropt2}

As before  $G\in L^\infty_{q\ts p}$ and $H$ is the Hankel operator defined by the co-analytic part of $G$.  Furthermore,  $\cM$ is a subspace of $H^2_p$ satisfying \eqref{cM1}. In this section we assume that $\|HP_\cM S\|< \ga_\cM=\|HP_\cM \|$. In other words, we assume that the defect operator $D_\cM^\circ$ defined by \eqref{defect2} is invertible. This additional condition allows us to simplify the formula for the central optimal solution to the $\cM$-restricted Nehari problem presented in Proposition \ref{propforopt1}. We shall prove the following theorem.

\begin{theorem}\label{T:main1}
Let $G\in L^\infty_{q\ts p}$, and let $\cM$ be a subspace of $H^2_p$ satisfying \eqref{cM1}.  Assume the defect operator $D_\cM^\circ$ defined by \eqref{defect2} is invertible, and put
\begin{equation}\label{LambdaM}
\La_\cM=D_\cM^{\circ-2} S^* D_\cM^2.
\end{equation}
Then $\spec(\La_\cM)\leq 1$, and  the central optimal solution $F_\cM$ to the $\cM$-restricted Nehari problem is given by $F_\cM=G-\Phi_\cM$, where $\Phi_\cM\in L^2 _{q\ts p}$  has finite $\cM$-norm, the co-analytic part of $\Phi_\cM$ is equal to $G_-$, and  the analytic part of $\Phi_\cM$ is given by
\begin{equation} \label{solform}
\Phi_{\cM,+}(\la) =\Pi H(I-\la \La_\cM)^{-1}\La_\cM E= N_\cM(\la) M_\cM(\la)^{-1} \quad (\la \in \BD),
\end{equation}
where
\begin{equation}\label{solform2}
N_\cM(\la)=\Pi H(I-\la S^*)^{-1}\La_\cM E,\quad
M_\cM(\la)=I-\la E^*(I-\la S^*)^{-1}\La_\cM E.
\end{equation}
In particular, $M(\la)$ is invertible for each $\la\in\BD$.
\end{theorem}

{The formulas in the above theorem  for the central optimal solution are inspired by the formulas for  the central suboptimal solution in Sections IV.3 and IV.4 of \cite{FFGK98}.}

We first  prove two lemmas.  In what follows   $P_\cM$ and  $R_\cM$ are  the orthogonal projections of  $H^2_p$ onto $\cM$ and $S^*\cM$, respectively, and $Q_\cM=SR_\cM$.

\begin{lemma}\label{lemPQR}Let  $\cM$ be a subspace of $H^2_p$ satisfying \eqref{cM1}.  Then
\begin{equation}\label{PQR}
 R_\cM=S^*P_\cM S, \quad R_\cM S^*=S^*P_\cM, \quad Q_\cM=P_\cM S.
 \end{equation}
 \end{lemma}
\begin{proof}[\bf Proof]  Note that
\[
 (S^*P_\cM S)^2=S^*P_\cM SS^*P_\cM S= S^*P_\cM S-S^*P_\cM(I- SS^*)P_\cM S.
\]
Since $I- SS^*$ is the orthogonal projection onto  $\kr S^*$, the second part of \eqref{cM1}  implies that $P_\cM(I- SS^*)=I- SS^*$.  Thus $ (S^*P_\cM S)^2= S^*P_\cM S$, and hence $S^*P_\cM S$ is an orthogonal projection. The range of this orthogonal  projection is $S^*\cM$, and therefore the first identity in \eqref{PQR} is proved.

Using this first identity  and $P_\cM(I- SS^*)=I- SS^*$ we see that
\[
R_\cM S^*=S^*P_\cM SS^*=S^*P_\cM-S^*P_\cM(I-SS^*)=S^*P_\cM.
\]
Thus the second identity in \eqref{PQR} also holds.  Finally,
\[
Q_\cM=SR_\cM=(R_\cM S^*)^*=( S^* P_\cM)^*=  P_\cM S.
\]
Thus \eqref{PQR} is proved. \end{proof}

\begin{lemma}\label{lemprojcF}
Let $G\in L^\infty_{q\ts p}$, and let $\cM$ be a subspace of $H^2_p$ satisfying \eqref{cM1}.  Assume the defect operator $D_\cM^\circ$ defined by \eqref{defect2} is invertible. Then the range $\cF$ of the operator $D_\cM Q_\cM$ is closed and  the orthogonal projection of  $H^2_p$ onto $\cF$ is given by
 \begin{equation}\label{projcF}
P_\cF= D_\cM Q_\cM D_{\cM}^{\circ -2}Q_\cM^* D_{\cM}.
\end{equation}
\end{lemma}
\begin{proof}[\bf Proof]
We begin with two identities:
\begin{equation}\label{interRP}
 D_\cM P_\cM=P_\cM D_\cM , \quad D_{\cM}^{\circ} R_\cM=R_\cM D_{\cM}^{\circ}.
\end{equation}
Since $P_\cM$ is an orthogonal projection,   the first equality in \eqref{interRP} follows directly from the definition of  $D_\cM$ in \eqref{defect1}. To prove the second, we use  the second identity in \eqref{PQR}. Taking adjoints and using the fact that $R_\cM$ and $P_\cM$ are orthogonal projections, we see that  $P_\cM S=SR_\cM$. It follows that $D_{\cM}^{\circ}$ is also given by
\begin{equation}\label{defect3}
 D_{\cM}^{\circ}= (\ga_\cM^2 I-R_\cM S^*H^*HSR_\cM )^{1/2}.
\end{equation}
From this formula for $D_{\cM}^{\circ}$ the second identity in \eqref{interRP} is clear.

Now assume that $D_{\cM}^{\circ}$ is invertible, and let $P$ be the operator defined by the right hand side of \eqref{projcF}.  Clearly, $P$ is selfadjoint. Let us prove that $P$  is a projection.  Using the second equality in \eqref{interRP} we have
\begin{align*}
 P^2&=  D_\cM Q_\cM D_{\cM}^{\circ -2}Q_\cM^* D_{\cM}^2
 Q_\cM D_{\cM}^{\circ -2}Q_\cM^* D_{\cM}
 \\
 &=D_\cM Q_\cM D_{\cM}^{\circ -2}(R_\cM S^*D_{\cM}^2 SR_\cM)D_{\cM}^{\circ -2}Q_\cM^* D_{\cM}
  \\
 &=D_\cM Q_\cM  R_\cM D_{\cM}^{\circ -2}(S^*D_{\cM}^2 S) D_{\cM}^{\circ -2} R_\cM Q_\cM^* D_{\cM}.
\end{align*}
Observe that $Q_\cM  R_\cM= S  R_\cM^2= S  R_\cM=Q_\cM$. Since $D_{\cM}^{\circ 2} = S^*D_{\cM}^2S$, it follows that
\[
P^2=D_\cM Q_\cM D_{\cM}^{\circ -2}Q_\cM^* D_{\cM}=P.
\]
 Thus $P$ is an orthogonal projection. This implies that $D_\cM Q_\cM$ has a closed range, and $P_\cF=P$. \end{proof}

\begin{proof}[\bf Proof of Theorem \ref{T:main1}] Our starting point is formula \eqref{eqPhiM+}.  Recall that $\om_1$ and $\om_2$ are zero on $\kr Q_\cM^*D_\cM$. From Lemma \ref{lemprojcF} we know that   $D_\cM Q_\cM$ has a closed range. It follows that $\om_1 =\om_1 P_\cF$ and $\om_2=\om_2 P_\cF$, where  $P_\cF$ is  the orthogonal projection of  $H^2_p$ onto $\cF= \im D_\cM Q_\cM$.  Using  the formula for $P_\cF$ given by \eqref{projcF}, the second intertwining relation in \eqref{interRP}, the identities in \eqref{PQR} and  the definition of $\om$ in \eqref {defoma},  \eqref {defoma} we compute
\begin{align*}
 \om_1 D_\cM&= \om_1 P_\cF D_\cM=\om_1 D_\cM Q_\cM D_{\cM}^{\circ -2}Q_\cM^* D_{\cM}^2
 \\
 &=\Pi H  R_\cM D_{\cM}^{\circ -2}R_\cM S^* D_{\cM}^2
= \Pi H  D_{\cM}^{\circ -2}R_\cM S^* D_{\cM}^2
   \\
 &=  \Pi H  D_{\cM}^{\circ -2} S^* R_\cM D_{\cM}^2=\Pi H \La_\cM P_\cM,
\end{align*}
and
\begin{align*}
 \om_2 D_\cM&= \om_2 P_\cF D_\cM= \om_2 D_\cM Q_\cM D_{\cM}^{\circ -2}Q_\cM^* D_{\cM}^2
  \\
 &=  D_\cM R_\cM D_{\cM}^{\circ -2}Q_\cM^* D_{\cM}^2= D_\cM\La_\cM P_\cM.
\end{align*}
Furthermore, using the intertwing relations in \eqref{interRP} and  the second identity in \eqref{PQR} we see that $R_\cM\La_\cM =\La_\cM P_\cM$. In particular, $\La_\cM$ leaves $\cM$ invariant.

Let us now prove that $\spec(\La_\cM )\leq 1$. Note that
\begin{align*}
\spec(\om_2)&= \spec(\om_2P_\cF)= \spec(D_\cM R_\cM D_\cM^{\circ-2}S^*D_\cM)
\\
&=\spec(R_\cM D_\cM^{\circ-2}S^*D_\cM^2)= \spec(R_\cM\La_\cM)= \spec(\La_\cM P_\cM).
\end{align*}
Thus  $ \spec(\La_\cM  P_\cM)\leq 1$, because $\om_2$ is contractive. Since $\La_\cM$ leaves $\cM$ invariant,  we see that relative to the orthogonal decomposition $H^2_p=\cM\oplus  \cM^\perp$ the operator $\La_\cM$ decomposes as
\begin{equation}\label{matLaM1}
 \La_\cM=\begin{bmatrix}P_\cM\La_\cM P_\cM&\star
  \\
 0& (I-P_\cM)\La_\cM(I- P_\cM)
\end{bmatrix}.
\end{equation}
Note that $(I- P_\cM)(I- R_\cM)=(I- P_\cM)$. Using the latter identity,  the formulas  \eqref{defect1} and  \eqref{defect3},
and the intertwining relations in \eqref{interRP}, we obtain
\begin{align*}
(I-P_\cM)\La_\cM(I- P_\cM)&=(I-P_\cM)(I- R_\cM)\La_\cM(I- P_\cM)
\\ &=(I-P_\cM)(I- R_\cM)S^*(I- P_\cM)
 \\
 &=(I-P_\cM)S^*(I- P_\cM).
\end{align*}
Thus $(I-P_\cM)\La_\cM(I- P_\cM)$ is a contraction. Hence  $\spec((I-P_\cM)\La_\cM(I- P_\cM)\leq 1$.  But then \eqref{matLaM1}  shows that $\spec(\La_\cM)\leq 1$.

Next, using that $\La_\cM \cM\subset S^* \cM\subset\cM$ and $\im E= \kr S^*\subset \cM$,
we obtain for each $\la\in\BD$ that
\begin{align*}
\Phi_{\cM, +}(\la)
&=\om_1(I-\la\om_2)^{-1}D_\cM E
 =\om_1 D_\cM (I-\la\La_\cM P_{\cM})^{-1}E
 \\
 &= \Pi H\La_\cM P_{\cM} (I-\la\La_\cM P_{\cM})^{-1}E
 =\Pi H (I-\la\La_\cM P_{\cM})^{-1}\La_\cM P_{\cM}E
 \\
 & =\Pi H (I-\la\La_\cM)^{-1}\La_\cM E,
\end{align*}
which gives formula \eqref{solform}.

Finally, to see that \eqref{solform2} holds, note that $\La_\cM S=I$. Hence $\La_\cM$ is a
left inverse of $S$. Since $E$ is an isometry with $\im E=\kr S^*$, we have
$\La_\cM=S^*+\La_\cM EE^*$. Therefore, for each $\la\in\BD$,
\begin{align*}
\Phi_{\cM, +}(\la)
&=\Pi H (I-\la \La_\cM)^{-1}\La_\cM E
=\Pi H (I-\la S^*-\la \La_\cM EE^*)^{-1}\La_\cM E
 \\
&=\Pi H (I-\la(I-\la S^*)^{-1} \La_\cM EE^*)^{-1}(I-\la S^*)^{-1}\La_\cM E
 \\
&=\Pi H (I-\la S^*)^{-1}\La_\cM E  (I-\la E^*(I-\la S^*)^{-1}\La_\cM E)^{-1}
 \\
&= N(\la)M(\la)^{-1}.
\end{align*}
In particular, $M(\la)$ is invertible.
\end{proof}

{\noindent\textbf{Remark.} From  $R_\cM \La_\cM=\La_\cM P_\cM$ we see that  $\La_\cM$  leaves $\cM$ invariant. Thus, if   $\cM$ in Theorem \ref{T:main1} is finite dimensional,  then   $\Phi_{\cM, +}$ in \eqref{solform} is a rational  function in $H^2_{p\ts q}$,  and hence  $\Phi_{\cM, +}$ is a rational  $p\ts q$ matrix function which has no pole in the closed unit disk.}\\


Next we present  a criterion in terms of maximizing vectors  under which Theorem \ref{T:main1} applies.

\begin{proposition}\label{propDMinv}
Assume $\rank HP_\cM$ is finite.  Then $D_\cM^\circ$ is invertible if and only if
none of  the maximizing vectors of  $HP_\cM$ belongs to $SH^2_p$.
\end{proposition}

\begin{proof}[\bf Proof]  A vector $h\in H^2_p$ is a maximizing vector of  $HP_\cM$ if and only if $0\not =h\in \cD_\cM^\perp$. Thus we have to show that
 invertibility of $D_\cM^\circ$  is equivalent to  $\cD_\cM^\perp\cap SH^2_p=\{0\}$.

Assume $\cD_\cM^\perp\cap SH^2_p\not=\{0\}$. Thus, using the definition of a maximizing vector,  there exists $Sv$ with $v\not=0$ such that $\|H P_{\cM}Sv\|=\ga_\cM\|Sv\|$.  Since $S$ is an isometry we see that $\|H P_{\cM}Sv\|=\ga_\cM\|v\|$. It follows that   $v$ is in the kernel of $D_\cM^\circ$,  and hence $D_\cM^\circ$ is not invertible.

Conversely, assume that $\cD_\cM^\perp \cap SH^2_p=\{0\}$. Note
that $\rank (H P_\cM S)$ is also finite. Hence $H P_\cM S$ has a maximizing vector,
say $v$. We may assume that $\|v\|=1$. By our assumption the vector $Sv$ is not
a maximizing vector of $H P_\cM$. Hence
\[
\|H P_\cM S\|=\|HP_\cM S \| \| v\|= \|HP_\cM S v\|
<\|HP_\cM\|\|Sv\|=\ga_\cM \|Sv\|=\ga_\cM.
\]
Therefore $D_\cM^{\circ 2}=\ga_\cM^2 I- S^*P_\cM H^* H P_\cM S$ is
positive definite, and thus invertible. Consequently, $D_\cM^{\circ}$ is invertible.
\end{proof}

For later purposes we mention the following. It is straightforward to prove that $D_\cM^\circ$ is invertible if and only if the operator  $\ga_\cM^2 I- HP_\cM SS^*P_\cM H^*$ is invertible, and in that case we have
\begin{align}
& \La_\cM P_\cM H^*= R_\cM H^* V_-^* (\ga_\cM^2 I- HP_\cM SS^*P_\cM H^*)^{-1}\ts\nonumber
\\
 &\hspace{7cm}\ts(\ga_\cM^2 I-HP_\cM H^*),\label{LaM2}
 \\
 & \La_\cM E=-R_\cM H^* V_-^* (\ga_\cM^2 I- HP_\cM SS^*P_\cM H^*)^{-1}HE.\label{LaM3}
\end{align}
These  formulas can be simplified further  using the operators $Z$ and $W$ associated to the Hankel operator $H$ which have been introduced at the end of Section \ref{S:intro}, see \eqref{defZ} and \eqref{defW}. Recall that  $\cX=\overline{\im H}$.  Since  $K_q^2\ominus \cX=\kr H^*$, the space $\cX$ is a reducing subspace for the  operators $\ga_\cM^2 I- HP_\cM SS^*P_\cM H^*$ and $\ga_\cM^2 I-HP_\cM H^*$. Furthermore,
\begin{align}
&\De_\cM:=(\ga_\cM^2 I- HP_\cM SS^*P_\cM H^*)|_\cX= \ga_\cM^2 I_\cX -ZWR_\cM W^*Z^*,\label{defDeM4}
\\
&\Xi_\cM: =(\ga_\cM^2 I-HP_\cM H^*)|_\cX=\ga_\cM^2 I_\cX-WP_\cM W^*. \label{defXiM4}
\end{align}
Note that $\De_\cM$ is invertible if and only if $D_\cM^\circ$ is invertible. Using the above operators, \eqref{LaM2} and \eqref{LaM3} can be written as
\begin{equation}\label{LaM4}
 \La_\cM P_\cM W^*= R_\cM W^*Z^*\De_\cM^{-1}\Xi_\cM,\quad
\La_\cM E=-R_\cM W^*Z^* \De_\cM^{-1}WE.
\end{equation}

\begin{corollary}\label{corNM4}
Let $G\in L^\infty_{q\ts p}$, and let $\cM$ be a subspace of $H^2_p$ satisfying \eqref{cM1}.  Assume the operator $\De_\cM$ defined by \eqref{defDeM4} is invertible. Then  the defect operator $D_\cM^\circ$ defined by \eqref{defect2} is invertible, and the functions $N_\cM$ and $M_\cM$ appearing in \eqref{solform2}
are also given by
\begin{align}
&N_\cM(\la)=N_{\cM,1}(\la)+N_{\cM,2}(\la), \label{SumForm14}\\
&\hspace{.6cm}N_{\cM,1}(\la)=  -\Pi HW^*(I-\la Z^*)^{-1}Z^*\De_\cM^{-1}WE\label{N14}\\
&\hspace{.6cm}N_{\cM,2}(\la)=  \Pi H(I-\la S^*)^{-1}(I-R_\cM)W^*Z^*\De_\cM^{-1}WE.\label{N24}
\end{align}
and
\begin{align}
&M_\cM(\la)=M_{\cM,1}(\la)+M_{\cM,2}(\la),  \label{SumForm24}\\
&\hspace{.6cm}M_{\cM,1}(\la)=I+\la E^*W^*(I-\la Z^*)^{-1}Z^*\De_\cM^{-1}
WE \label{M14}\\
&\hspace{.6cm}M_{\cM,2}(\la)=-\la E^* (I-\la S^*)^{-1}(I-R_\cM)W^*Z^*\De_\cM^{-1}WE. \label{M24}
\end{align}
Furthermore, if $\spec(Z^*\De_\cM^{-1}\Xi_\cM)<1$, then $M_{\cM,1}(\la)$ is invertible for   $|\la|\leq 1$ and
\begin{equation}\label{invM14}
M_{\cM,1}(\la)^{-1}=I-\la E^*W^*(I-\la Z^*\De_\cM^{-1}\Xi_\cM)^{-1}Z^*\De_\cM^{-1}WE, \quad |\la|\leq 1.
\end{equation}
\end{corollary}

\begin{proof}[\bf Proof]  For operators $A$ and $B$ the invertibility of $I+AB$ is equivalent to  the invertibility of  $I+BA$.  Using this fact  it is clear that the invertibility of $D_\cM^\circ$  follows form the invertibility of  $\De_\cM$.  Hence we can apply Theorem \ref{T:main1}. Writing $R_\cM$ as $I-(I-R_\cM)$ and using \eqref{LaM4} we see that \eqref{SumForm14} holds with $N_{\cM,2}$ being given by \eqref{N24} and with
\begin{equation}\label{N14a}
 N_{\cM,1}(\la)=  -\Pi H(I-\la S^*)^{-1}W^*Z^*\De_\cM^{-1}WE.
\end{equation}
The intertwining relation  $WS=ZW$ yields  $(I-\la S^*)^{-1}W^*=W^*(I-\la Z^*)^{-1}$. Using the latter identity in  \eqref{N14a} yields  \eqref{N14}.  In a similar way one proves  the identities \eqref{SumForm24}-\eqref{M24}.

To complete the proof  assume  $\spec(Z^*\De_\cM^{-1} \Xi_\cM)<1$. Then the inversion formula for $M_{\cM, 1}(\la)$ follows
from the standard inversion formula from \cite[Theorem~2.2.1]{BGKR}, where
we note that the state operator in the inversion formula equals
\begin{align*}
Z^*-Z^*\De_\cM^{-1} WEE^*W^*
&=Z^*\De_\cM^{-1}(\ga_\cM^2 I-ZWR_\cM W^*Z^*-WEE^*W^*)\\
&=Z^*\De_\cM^{-1}(\ga_\cM^2 I -W(SR_\cM S^* +EE^*)W^* )\\
&=Z^*\De_\cM^{-1}(\ga_\cM^2 I -W(SS^*P_\cM +EE^*P_\cM)W^* )\\
&=Z^*\De_\cM^{-1}(\ga_\cM^2 I -WP_\cM W^* )=Z^*\De_\cM^{-1}\Xi_\cM ,
\end{align*}
as claimed. Here we used the second identity in \eqref{PQR}, and the fact that $P_\cM E=E$, because  $\im E=\kr S^*\subset \cM$.\end{proof}

\setcounter{equation}{0}
\section{The special case where $\cM=H^2_p$}\label{S:NehariCase}

Throughout this section $\cM=H^2_p$, that is, we are dealing with the $H^2_p$-restricted Nehari problem, which is just  the usual Nehari problem.   Since  $\cM=H^2_p$, we will surpress the index $\cM$ in our notation, and just write $D$, $D^\circ$, $\cD$, $\cD^\circ$, $\La$, etc.\ instead of $D_\cM$, $D_\cM^\circ$, $\cD_\cM$, $\cD_\cM^\circ$, $\La_\cM$, etc. In particular,
\begin{equation}\label{defect4}
\ga =\|H\|, \quad D=(\ga^2 I- H^*H )^{1/2}, \quad D ^\circ=(\ga ^2 I-S^*  H^*H  S)^{1/2}.
\end{equation}
We shall assume (cf., the first paragraph of Section \ref{S:MainRes}) that the following two conditions are satisfied
\begin{itemize}
 \item[(C1)] $H$ has finite rank,
 \item[(C2)] none of  the maximizing vectors of $H$ belongs $SH_p^2$, and  the space spanned by the maximizing vectors of $H$ has
 dimension $p$.
 \end{itemize}
Note that the space spanned by the maximizing vectors of $H$ is equal to $\kr D=\cD^\perp$, where $\cD$ is the closure of the range of $D$.  As $H^2_p=\kr S^*\oplus SH^2$,  we see that
\begin{equation}\label{equiC2}
\mbox{(C2)} \quad \Longleftrightarrow\quad  H^2_p=\kr D\dot{+}SH^2_p \quad \Longleftrightarrow\quad H^2_p=\kr S^* \dot{+}\cD.
\end{equation}
Here $ \dot{+}$ means direct sum, not necessarily orthogonal direct sum.

Let  $Z$ and $W$ be the operators defined by \eqref{defZ} and \eqref{defW}, respectively, and
\begin{equation}\label{DeXi}
 \De= \ga^2 I_\cX-ZWW^*Z^*, \quad \Xi= \ga^2 I_\cX- WW^*.
\end{equation}
We shall prove the following theorem.

\begin{theorem}\label{T:NehariCase}
Let $G\in L^2_{q\ts p}$, and assume that the Hankel operator $H$ associated with the co-analytic part of $G$ satisfies   conditions \textup{(C1)}  and \textup{(C2)}. Then  the operator $\De$ defined by the first identity in  \eqref{DeXi} is invertible and the Nehari problem associated with $G$ has a unique optimal solution ${F}\in H^\iy_{q\ts p} $. Moreover, this unique solution is given by   ${F}=G_+ -{\Phi}_+$, where  $G_+$ is the analytic part of $G$ and ${\Phi}_+$ is the rational $q\ts p$ matrix-valued $H^\iy$ function given by
\begin{align*}
&{\Phi}_+(\la)=N(\la)M(\la)^{-1},  \mbox{  where}
\\
&\hspace{.5cm}N(\la)=-\Pi H W^*(I_\cX-\la Z^*)^{-1}Z^*\De^{-1}WE,
\\
&\hspace{.5cm}M(\la)=I+\la E^*W^*(I_\cX-\la Z^*)^{-1}Z^*\De^{-1}WE,
\end{align*}
Furthermore, $r_{spec}(Z^*\De^{-1}\Xi)<1$, and the inverse of $M(\la)$ is given by
\[
M(\la)^{-1}=I-\la E^*W^*(I_\cX-\la Z^*\De^{-1}\Xi)^{-1}Z^*\De^{-1}WE.
\]
Here $\Xi$ is the operator  defined by the second identity in  \eqref{DeXi}.
\end{theorem}

{The fact that condition (C2) implies uniqueness of the optimal solution follows from \cite{AAK71}; cf.,  Theorem 7.5 (2) in \cite{AD08}.  It will} be convenient first to prove the following lemma.

\begin{lemma}\label{L:F=D}
Assume  $H$ is compact and \textup{(C2)} is satisfied. Then the following holds.
\begin{itemize}
 \item[(i)] The operator $D^\circ$ is invertible, and the range of $DS$ is closed and is equal to $\cD$.  In particular, the optimal  solution to the Nehari problem is unique.
 \item[(ii)] The  subspace $\kr D=\cD^\perp$ of $H^2_p$ is cyclic for $S$.
 \item[(iii)] The operators  $\om_2=DD^{\circ -2}S^*D$ and $\La=D^{\circ -2}S^*D$   are well-defined and  strongly stable.
 \end{itemize}
 \end{lemma}

\begin{proof}[\bf Proof]  We split the proof  into three parts according to the three items.

\smallskip\noindent\textsf{Part 1.}  We prove (i).  Since $H$ is compact, the selfadjoint operator $D$ has closed range and a finite dimensional null space. Thus $D$ is a Fredholm operator of index zero.  See \cite[Section XI.1]{GGK90} for the definitions of these notions.   Note  $S$ is a Fredholm operator of index $p$.  Thus $DS$ is also a Fredholm operator. In particular, the range of $DS$ is closed, and hence $\cF: =\ov{D SH^2_p}= D SH^2_p$. Moreover,
\[
\tu{ind}(DS)=\tu{ind}(D)+\tu{ind}(S)= -p.
\]
Here $\tu{ind}$ denotes the index of a Fredholm operator, and we used the fact (\cite[Theorem XI.3.2.]{GGK90}) that the index of a product of two Fredholm operators is the sum of the indices of the factors. On the other hand, since $\kr D\cap SH^2_p$ consists of the zero vector only,  we see that $\kr DS=\{0\}$, and hence, using the definition of the index,  we have $p= \codim DS H^2_p$. But
$DS H^2_p\subset D H^2_p=\cD$ and, by the third part of \eqref{equiC2}, we have $\codim \cD=p$  Thus $\cF=\cD$. The latter implies that the central solution of \eqref{rcl3}  is the only optimal solution of the  Nehari problem; see the remark made at the end of Section \ref{secRestrN}.

Finally, $\kr DS=\{0\}$ and $DS$ has closed range, yields $D^{\circ 2}=S^*D^2S$ is invertible.  This completes the proof of (i).

\smallskip\noindent\textsf{Part 2.}  We prove (ii).
We begin with a remark. From (i) we know that that $D^{\circ}$ is invertible. Thus the operators  $\om_2=DD^{\circ -2}S^*D$ and $\La=D^{\circ -2}S^*D^2$ are well defined. Clearly, $\om_2D=D\La$, and hence $\om_2^k D=D\La^k$ for $k=0, 1,2 , \ldots$. It follows that
\begin{equation}\label{om2La}
\La^{k+1}=D^{\circ -2}S^*D^2 \La^k=D^{\circ -2}S^*D\om_2^k , \quad k=0, 1,2 , \ldots.
\end{equation}
Since $\om_2$ is a contraction, we conclude that  $\sup_{k\geq 0} \|\La^{k}\|<\iy$.

Our aim is to prove that  $H^2_p=\bigvee_{k=0}^\iy S^k \cD^\perp$.  Take  $h\in H^2_p$ perpendicular to $\bigvee_{k=0}^\iy S^k \cD^\perp$. The latter is equivalent to  $S^{*k}h$ being  perpendicular to $\cD^\perp$ for $k=0,1,2,\ldots$, that is, $S^{*k}h\in \cD$ for $k=0,1,2,\ldots$. Recall that the range of $D$ is closed, because $H$ is compact. Thus for each $k=0,1,2,\ldots$ the vector $S^{*k}h=D h_k$ for some $h_k\in \cD$.  Thus $S^{*k+1}h=S^*Dh_k$. Since $D^{\circ}$ is invertible,  Lemma \ref{lemprojcF} specified for the case $\cM=H^2_p$ tells us that $P:=
DSD^{\circ -2}S^*D$ is the orthogonal projection of $H^2_p$ onto $\cF=\cD=\im D$. Thus for $k=0,1,2,\ldots$ we have
\[
S^{*k}h=Dh_k=DPh_k=D^2S D^{\circ -2}S^*D h_k= D^2S D^{\circ -2} S^{*k+1}h=\La^*S^{*k+1}h,
\]
and by induction $h=\La^{*k}S^{*k}h$.  Since $\lim_{k\to 0} \|S^{*k+1}h\|=0$, and $\sup_{k\geq 0} \|\La^{k}\|<\iy$, it follows that $\|h\|=0$. Hence $h=0$,  and
we can conclude that $\bigvee_{k=0}^\iy S^k \cD^\perp=H^2_p$.  This proves (ii).

\smallskip\noindent\textsf{Part 3.}  We prove (iii).  We already know that $\om_2$ and $\La$ are well defined. We first prove that $\om_2$ is strongly stable, that is, $\lim_{k\to\iy}\om_2^k v=0$ for any $v\in H^2_p$. Note that
$\om_2 DS=D$. Hence $\om_2^k D S^k=D$ for $k=0,1,2,\ldots$. Since $\cD^\perp=\kr D$,
we have for any nonnegative integers $k,l$ that
$\om_2^{k+l} D S^k \cD^\perp=\om_2^l S\cD^\perp=0$. In other words, the kernel of $\om_2^k$
includes $\cX_k:=\bigvee_{\nu=0}^k S^\nu\cD^\perp$.
Let $v\in H^2_p$. According  to (ii), we have
$\bigvee_{\nu=0}^\infty S^\nu\cD^\perp=H^2_p$. Thus $P_{\cY_k}v\to 0$, with $\cY_k=H^2_p\ominus \cX_k$,
and since $\om_2$ is contractive, we find that
\[
\|\om_2^k v\|=\|\om_2^kP_{\cY_k}v\|\leq \|P_{\cY_k}v\|\to0.
\]
Thus $\om_2$ is strongly stable, as claimed,  {and the fact that $\La$ is strongly stable follows immediately from \eqref{om2La}.} \end{proof}


\begin{proof}[\bf Proof of Theorem \ref{T:NehariCase}]
From  Lemma \ref{L:F=D} (i) we know that $D^\circ$ is invertible,
and the   optimal solution is unique. Since the invertibility of  $D^\circ$ implies the invertibility of $\De$,  we can apply Theorem \ref{T:main1} and Corollary \ref{corNM4} with $\cM=H^2_p$ to get the desired formula for $\Phi_+$. Note that $R_{H^2_p}=I$, and hence in this case the functions appearing in \eqref{N24} and \eqref{M24} are identically zero.

Put $T=Z^*\De\Xi$.  Next we show that $\spec(T)<1$. By specifying the first identity in \eqref{LaM4} we see that   $\La W^*=W^* T$, and thus $\La^k W^*=W^* T^k$ for each $k\in \BN$. Since $\La$ is strongly stable (by Lemma \ref{L:F=D} (iii)), we arrive at $\lim_{k\to \iy} W^*T^k x=0$. The fact that $H$ has finite rank, implies that the range of $H$ is closed, and hence $W$ is surjective. But then $(WW^*)^{-1}W$ is a left inverse of $W^*$, and $T^kx=(WW^*)^{-1}WT^kx\to 0$ if $k\to \iy$. Thus $T$ is strongly stable. Since the underlying space $\cX$ is finite dimensional, we conclude  that $\spec(Z^*\De\Xi)=\spec(T)<1$.

Finally, since $\spec(Z^*\De\Xi)=\spec(T)<1$,  the invertibility of $M(\la)$ for $\|\la\|$ and the formula for its inverse   follow by specifying the final part of Corollary \ref{corNM4} for the case when $\cM=H^2_p$.
\end{proof}

\setcounter{equation}{0}
\section{Convergence of central optimal  solutions}\label{secprfmthm}

Throughout  $G\in L^\iy_{q\ts p}$ and  $H$ is  the Hankel operator
defined by the co-analytic part of $G$.  We assume that conditions (C1) and (C2)
formulated in the first paragraph of Section \ref{S:MainRes} are satisfied. Furthermore, $\cM_0$ is a finite dimensional $S^*$-invariant subspace of $H^2_p$, and $\cM_0, \cM_1, \cM_2, \ldots$  is a sequence of subspaces of $H^2_p$ defined recursively by \eqref{cMk}. We set  $P_k=P_{\cM_k}$.  {From the remarks made in the paragraph preceding Theorem \ref{T:main2} one sees that
\begin{equation}\label{ProForm}
{I-P_k=S^k(I-P_0)S^{*k},\quad S^*P_{k}=P_{k-1} S^*\quad P_{k}E=E.}\quad (k\in \BN).
\end{equation}
Here $E$ is the embedding operator defined by  \eqref{defE}.}

In this section we will proof Theorem \ref{T:main2}. In fact we will
show that with an appropriate choice of the initial space $\cM_0$ convergence
occurs at an ever faster rate than stated in Theorem \ref{T:main2}.
We start with a lemma that will be of help when proving the increased
rate of convergence.

\begin{lemma}\label{L:SpecBound}
Let $Z$ and $W$ be the operators defined by \eqref{defZ} and \eqref{defW}, respectively, and put  $\cX_0=W\cM_0^\perp \subset\cX$.  Then $\cX_0$ is $Z$-invariant of $\cX=\im W$, and $\spec(Z_0)\leq\spec(Z)$. Furthermore,  let the operators $Z_0: \cX_0\to  \cX_0$ and  $W_0: H^2_p\to  \cX_0$ be defined by $Z_0=Z|_{\cX_0}$ and $W_0=\Pi_{\cX_0}W$, where  $\Pi_{\cX_0}$ is the orthogonal projection of $\cX$ onto $\cX_0$. Then
\begin{equation}\label{ZtoZ0}
Z^kW(I-P_0)=\Pi_{\cX_0}^*Z_0^k W_0(I-P_0),\quad k=0,1,2,\ldots.
\end{equation}
\end{lemma}

\begin{proof}[\bf Proof]
Since $ZW=WS$ and $\cM_0^\perp$ is invariant under $S$, we see that $\cX_0$ is invariant under $Z$,  and thus  $\spec(Z_0)\leq\spec(Z)$. From the definition of $Z_0$ and $W_0$ we see that $Z\Pi_{\cX_0}^*=\Pi_{\cX_0}^*Z_0$ and $\Pi_{\cX_0}^*W_0(I-P_0)=W(I-P_0)$. Thus
\[
Z^kW(I-P_0)=Z^k \Pi_{\cX_0}^*W_0(I-P_0) =\Pi_{\cX_0}^* Z_0^k W_0(I-P_0), \quad k=0,1,2,\ldots
\]
This proves \eqref{ZtoZ0}.\end{proof}

Assume $0<r<1$ such that the poles of $G$ inside $\BD$ are in the
open disc $\BD_r$. As mentioned in the introduction, the poles of
$G$ inside $\BD$ coincide with the eigenvalues of $Z$. Thus $\spec(Z)<r$.
By Lemma \ref{L:SpecBound}, $\spec(Z_0)\leq\spec(Z)<r$.  {In what follows we fix $0<r_0<1$  such that $\spec(Z_0)<r_0<r$. We will show that the convergence of the central optimal solutions $F_k$ in Theorem \ref{T:main2} is proportional to~$r_0^{k}$. }

{For simplicity, we will adapt}  the notation of Section
\ref{S:NehariCase}, and write $\ga$, $\De$, $N$ and $M$ instead of
$\ga_{H^2_p}$, $\De_{H^2_p}$, $N_{H^2_p}$ and $M_{H^2_p}$.
Futhermore, we  use the abbreviated notation $P_k$, $\ga_k$,
$\La_k$, $\Xi_k$, and $\De_k$ for the operators $P_{\cM_k}$, $\ga_{\cM_k}$,
$\La_{\cM_k}$, $\Xi_{\cM_k}$, and $\De_{\cM_k}$ appearing in Section \ref{S:centropt2} for $\cM=\cM_k$.

As a first step towards the proof of our convergence result  we prove the following lemma.

\begin{lemma}\label{L:DeltaInv}
{Assume conditions \textup{(C1)} and \textup{(C2)}  are satisfied. Then $\De_k\to_{r_0^2}\De$, and for $k\in\BN$ large enough  $\De_k$ is invertible, and $\De_k^{-1}\to_{r_0^2}\De^{-1}$.}
\end{lemma}

\begin{proof}[\bf Proof]
{We begin with a few remarks. Recall that  for $\cM$ in \eqref{cM1}  the operator $R_{\cM}$ is defined to be the orthogonal projection of $H^2_p$ onto $S^*\cM$; see the first paragraph of Section \ref{secRestrN}.  For   $\cM=\cM_k$  we have $S^*\cM_k=\cM_{k-1}$ by  \eqref{cMk}, and thus $\cM=\cM_k$ implies  $R_{\cM_{k}}=P_{k-1}$. It follows  that the operator $\De_k$ is given by $\De_k=\ga_k^2 I_\cX- ZWP_{k-1}W^*Z^*$; c.f., the second part of \eqref{defDeM4}. From  the invertibility of $D_{\cM_k}^\circ$ we obtain that $\De_k$ is invertible as well; see the first paragraph of the proof Corollary \ref{corNM4}. The identities in \eqref{LaM4} for $\cM=\cM_k$ now take the form}
\begin{equation}\label{rels}
\La_k P_k W^*=P_{k-1}W^*Z^*\De_k^{-1}\Xi_k,\quad
\La_kE=-P_{k-1}W^*Z^*\De_k^{-1}WE.
\end{equation}

Observe that
$
\ga_k^2=\|HP_k\|^2=\|P_k H^*\|^2=\spec(HP_k H^*)=\|HP_k H^*\|.
$
By a similar computation $\ga^2=\|HH^*\|$. Thus, using
\eqref{ProForm} and \eqref{ZtoZ0},
\begin{align*}
|\ga^2-\ga_k^2|
&=|\|HH^*\|-\|HP_k H^*\||
\leq \|HH^*-HP_k H^*\|\\
& =\|HS^k (I-P_0) S^{*k}H^*\|
=\|Z_0^k W_0(I-P_0)W_0^* Z_0^{*k}\|\\
&\leq \|Z_0^k\|\,\| H\|\,\|(I-P_0)\|\,\|H^*\| \|Z_0^{*k}\|
=\| H\|^2\,\|Z_0^k\|^2.
\end{align*}
It follows that $\ga_k^2\to_{r_0^2}\ga^2$.
Next, again by \eqref{ProForm} and \eqref{ZtoZ0}, we obtain
\begin{align*}
\De_k
&=\ga_k^2 I-ZWP_{k-1}W^*Z^*\\
&=\ga_k^2 I-ZWW^*Z^*+ZWS^{k-1}(I-P_0)S^{*k-1}W^*Z^*\\
&=\De +(\ga_k^2 -\ga^2)I+P_{\cX_0} Z_0^{k-1}W_0(I-P_0)W_0Z_0^{*k}P_{\cX_0}.
\end{align*}
Clearly the second and third summand converge to zero proportional
to $r_0^{2k}$, and thus we may conclude that $\De_k\to_{r_0^2}\De$.

{Since $\De$ is invertible by Theorem \ref{T:NehariCase}. The result of the previous paragraph  implies that for $k$ large enough $\De_k$ is invertible and $\|\De_k^{-1}\|<L$ for some $L>0$ independent of $k$. Consequently $\De_k^{-1}\to_{r_0^2}\De^{-1}$.}
\end{proof}

\begin{proof}[\bf Proof of Theorem \ref{T:main2} (with $r_0^k$-convergence)]
We split the proof into  {four} parts. Throughout  $k\in \BN$ is assumed to be large enough so that $\De_k$ is invertible; see Lemma~\ref{L:DeltaInv}.\smallskip

\noindent\textit{Part 1.} Let $N$ and $M$ be as in Theorem \ref{T:NehariCase}. Put
\begin{align}
 N_{k,1}(\la)&=  -\Pi HW^*(I-\la Z^*)^{-1}Z^*\De_k^{-1}WE,\label{defNk1}\\
M_{k,1}(\la)&= I+\la E^*W^*(I-\la Z^*)^{-1}Z^*\De_k^{-1} WE.\label{defMk1}
\end{align}
Since the only dependence on $k$ in $N_{k,1}$ and $M_{k,1}$ occurs in the form of $\De_k$, it follows from Lemma \ref{L:DeltaInv}  that
\begin{equation}\label{convNMk1}
M_{k,1}\to_{r_0^2}M\ands N_{k,1}\to_{r_0^2}N.
\end{equation}

\noindent\textit{Part 2.} From Corollary \ref{corNM4} we know that
\begin{align}
 N_k(\la)&=N_{k,1}(\la)+N_{k,2}(\la), \quad N_{k,2}(\la)=\Pi H\Ga_k(\la)Z^*\De_k^{-1}WE,\label{defNk2}\\
 M_k(\la)&=M_{k,1}(\la)+M_{k,2}(\la), \quad M_{k,2}(\la)=-\la E^*\Ga_k(\la)Z^*\De_k^{-1}WE.\label{defMk2}
\end{align}
Here $\Ga_k(\la)=(I-\la S^*)^{-1}(I-P_{k-1})W^*$.  In this part we show that $M_{k,2}\to_{r_0} 0$.

Using the first identity in \eqref{ProForm}, the intertwining relation $ZW=WS$, and \eqref{ZtoZ0} we see that
\begin{align*}
\Ga_k(\la)&=(I-\la S^*)^{-1}S^{k-1}(I-P_{0})S^{*k-1}W^*\\
&=(I-\la S^*)^{-1}S^{k-1}(I-P_{0}) W_0^*Z_0^{* k-1}\Pi_{\cX_0}.
\end{align*}
Next we use that
\[
(I-\la S^*)^{-1}S^{k-1}=\sum_{j=0}^{k-2}\la^j S^{k-1-j} +\la^{k-1}(I-\la S^*)^{-1}.
\]
Thus $\Ga_k(\la)=\Ga_{k, 1}(\la)+\Ga_{k, 2}(\la)$, where
\begin{align*}
 \Ga_{k, 1}(\la)&= \Big(\sum_{j=0}^{k-2}\la^j S^{k-1-j}\Big)(I-P_{0}) W_0^*Z_0^{* k-1}\Pi_{\cX_0},\\
\Ga_{k, 2}(\la)&=\la^{k-1}(I-\la S^*)^{-1} (I-P_{0}) W_0^*Z_0^{* k-1}\Pi_{\cX_0}.
\end{align*}

 Now recall that $\cM_0$ is  $S^*$-invariant, and write $S_0=P_0 S P_0=P_0 S$. The fact that $\cM_0$ is finite dimensional implies $\spec(S_0)<1$. The computation
\begin{align*}
(I-\la S^*)^{-1}(I-P_0)W_0^*
&=(I-\la S^*)^{-1}W_0^*-(I-\la S^*)^{-1}P_0W_0^*\\
&=W_0^*(I-\la Z_0^*)^{-1}-(I-\la S_0^*)^{-1}P_0 W_0^*,
\end{align*}
shows that $(I-\la S^*)^{-1}(I-P_0)W_0^*$ is uniformly bounded on $\BD$. Since $\spec(Z_0)<r_0<1$, we conclude that $\Ga_{k,2}\to_{r_0} 0$.

Next observe that $E^* \big(\sum_{j=0}^{k-2}\la^j S^{k-1-j}\big)=0$, and thus
$E^*\Ga_{k,1}(\la)=0$ for each $k\in \BN$. We conclude that

\[
M_{k,2}(\la)= -\la E^*\Ga_{k,2}(\la)Z^*\De_k^{-1}WE.
\]
But then $\Ga_{k,2}\to_{r_0} 0$ implies that  the same holds true for  $M_{k,2}$, that is,  $M_{k,2}\to_{r_0} 0$. Indeed, this follows from the above identity and  the fact that the sequence $\De_k^{-1}$ is uniformly bounded. \smallskip

\noindent\textit{Part 3.} In this part we show that $N_{k,2}\to_{r_0} 0$. To do this we first observe that
\[
\Pi H S^{k-1-j}=\Pi  V_-^{k-1-j} H=\Pi  V_-^{k-1-j}P_\cX W=\Pi P_\cX Z^{k-1-j} W.
\]
Post-multiplying this identity with $I-P_0$ and using \eqref{ZtoZ0} yields
\[
\Pi H S^{k-1-j}(I-P_0)=\Pi P_{\cX_0}Z_0^{k-1-j}W_0(I-P_0).
\]
It follows that
\begin{align}
 N_{k,2}(\la)&=  \Big(\sum_{j=0}^{k-2}\la^j\Pi P_{\cX_0}Z_0^{k-1-j} \Big)W_0(I-P_0)W_0^* Z_0^{* k-1}+\nonumber\\
 &\hspace{4.5cm}+ \Pi H\Ga_{k, 2} (\la)Z^*\De_k^{-1}WE.\label{indNk2}
\end{align}
From the previous part of the proof  we know that $\Ga_{k,2}\to_{r_0} 0$, and by  Lemma \ref{L:DeltaInv}  the sequence  $\De_k^{-1}$ is uniformly bounded. It follows that the second term in the right hand side of \eqref{indNk2} converges to zero with a rate proportional to  $r_0^k$. Note that for $\la \in \BD$ we have
\[
\|\sum_{j=0}^{k-2}\la^j\Pi P_{\cX_0}Z_0^{k-1-j}\|\leq \sum_{j=0}^{k-2}  \|Z_0\|^{k-1-j} \leq\sum_{j=1}^\infty \|Z_0^j\|\leq \frac{L_0r_0}{1-r_0}.
\]
 Since $\spec(Z_0)<r_0<1$, we also have $\|Z_0^{* k-1}\| \to_{r_0} 0$. It follows that  the first term in the right hand side of  \eqref{indNk2} converges to zero with a rate proportional to  $r_0^k$.  We conclude that $N_{k,2}\to_{r_0} 0$.  \smallskip

\noindent\textit{Part 4.} To complete the proof, it remains to show that
$M_{k}^{-1}(\la)\to_{r_0}M^{-1}(\la)$ uniformly on $\ov{\BD}$. By similar
computations as in the proof of Lemma \ref{L:DeltaInv}, it follows that
$\Xi_k\to_{r_0^2}\Xi$. Hence $Z^*\De_k^{-1}\Xi_k\to_{r_0^2}Z^*\De^{-1}\Xi$.
By Theorem \ref{T:NehariCase} we have
$\spec(Z^*\De^{-1}\Xi)<1$. Thus for $k$ large enough also $\spec(Z^*\De_k^{-1}\Xi_k)<1$,
and $M_{k,1}(\la)$ is invertible on $\ov{\BD}$. From the fact that $M_{k,1}\to_{r_0^2} M$,
we see that $M_{k,1}^{-1}\to_{r_0^2} M^{-1}$, with $M_{k,1}^{-1}$ and $M^{-1}$ indicating here the functions on $\ov{\BD}$ with values $M_{k,1}(\la)^{-1}$ and $M(\la)^{-1}$ for each $\la\in \ov{\BD}$. In particular, the functions $M_{k,1}^{-1}$ are uniformly bounded on $\ov{\BD}$ by a constant independent of $k$, which implies
\[
I+M_{k,1}^{-1}M_{k,2}\to_{r_0} I,\quad
(I+M_{k,1}^{-1}M_{k,2})^{-1}\to_{r_0} I.
\]
As a consequence
\begin{align*}
M_k^{-1}
&=(M_{k,1}+M_{k,2})^{-1}=(I+M_{k,1}^{-1}M_{k,2})^{-1}M_{k,1}^{-1}
\to_{r_0} I\cdot  M^{-1}=M^{-1},
\end{align*}
which completes the proof.
\end{proof}

\noindent\textbf{Concluding remarks} \\Note  that the functions $M_{k,1}$ and $N_{k,1}$ given by  \eqref{defNk1} and \eqref{defMk1}  converge with a rate proportional to $r_0^{2k}$ rather than $r_0^k$; cf., \eqref{convNMk1}.  Consequently the same holds true for $M_{k,1}^{-1}$. Thus a much faster convergence may be achieved when $N_{k,1}M_{k,1}^{-1}$ are used instead of $N_kM_k^{-1}$. However, for the inverse of $M_{k,1}$ to exist on $\ov{\BD}$ we need $k$ to be large enough to guarantee $\spec(Z^*\De_k\Xi_k)<1$, and  it is at present not clear how large $k$ should be.\smallskip

\noindent {For the scalar case condition (C2) is rather natural. Indeed (see the second paragraph of Section \ref{S:MainRes}) for the scalar case condition (C2) is equivalent to the requirement that the largest singular value of the Hankel operator is simple.  The latter condition also appears in model reduction problems. In the matrix-valued case (C2) seems  rather special. We expect that a version of Theorem \ref{T:main2} can be proved by only using the first part of (C2), that is, by assuming that none of the maximizing vectors of the Hankel operator  belongs to $SH^2_p$; cf., Proposition \ref{propDMinv}.  However, note  that in that case the optimal solution of the Nehari problem may not be unique. }
\smallskip

{ Computational examples show that it may happen that the approximations of the optimal solution to the Nehari problem considered in this paper oscillate to the optimal solution when the initial space $\cM_0=\{0\}$.  Although the rate of convergence can be improved considerably by choosing a different initial space $\cM_0$, the same examples  show that the approximations  still oscillate  in much the same way as before to the optimal solution.  This suggests that approximating  the optimal solution may not be practical in some problems. In this case,  one may have to adjust these approximating optimal solutions.  We plan to return to this phenomenon in a later paper.} \smallskip

\noindent\textbf{Acknowledgement.} {The authors thank Joe Ball for mentioning the Helton-Young paper \cite{HY90} to  the second author.}


\end{document}